\documentclass[a4paper,reqno,10pt]{amsart}

\usepackage{hyperref}%\usepackage[colorlinks=true, linkcolor=blue, citecolor=green]{hyperref}
\usepackage{a4wide}

\usepackage{amsfonts}
\usepackage{bbm}

\usepackage{amssymb}
\usepackage{amstext}
\usepackage{amsmath}
\usepackage{amscd}
\usepackage{amsthm}
\usepackage{amsfonts}

\usepackage{graphicx}
\usepackage{latexsym}
\usepackage{mathrsfs}

\usepackage{enumitem}
\setlist[enumerate,1]{label={\upshape(\arabic*)}}
\setlist[enumerate,2]{label={\upshape(\alph*)}}

\numberwithin{equation}{section}
\usepackage{titlesec}
\titleformat{\subsection}{\normalfont\bfseries\upshape}{\thesubsection}{1em}{}  
\titlespacing*{\subsection}{0pt}{3ex}{2ex}  %
\usepackage{indentfirst}

\usepackage{tikz}
\usetikzlibrary{cd,arrows,matrix,backgrounds,positioning,calc,decorations.markings,decorations.pathmorphing,decorations.pathreplacing}
\tikzset{black/.style={circle,fill=black,inner sep=3pt,outer sep=3pt},
	white/.style={circle,fill=white,draw=black,inner sep=3pt,outer sep=3pt},
}

\usepackage{array}
\newcolumntype{C}{>{$}c<{$}}

\usepackage{makecell}
\usepackage{adjustbox}

\usepackage{xy}
\xyoption{all}

\newcommand{\Hom}{\operatorname{Hom}\nolimits}

\newcommand{\Iso}{\operatorname{Iso}\nolimits}

\theoremstyle{plain}
\newtheorem{theorem}{\bf Theorem}[section]
\newtheorem*{theoremi}{Theorem A}
\newtheorem{lemma}[theorem]{\bf Lemma}
\newtheorem{corollary}[theorem]{\bf Corollary}
\newtheorem{proposition}[theorem]{\bf Proposition}

\theoremstyle{definition}
\newtheorem{definition}[theorem]{\bf Definition}
\newtheorem{remark}[theorem]{\bf Remark}
\newtheorem{example}[theorem]{\bf Example}

\newtheorem{condition}[theorem]{\bf Condition}

\newtheorem*{conv}{Conventions and notation}
\newtheorem*{org}{Organization}

\newcommand{\bt}{\begin{theorem}}
	\newcommand{\et}{\end{theorem}}
\newcommand{\bl}{\begin{lemma}}
	\newcommand{\el}{\end{lemma}}
\newcommand{\bd}{\begin{definition}}
	\newcommand{\ed}{\end{definition}}
\newcommand{\bco}{\begin{condition}}
	\newcommand{\eco}{\end{condition}}
\newcommand{\bc}{\begin{corollary}}
	\newcommand{\ec}{\end{corollary}}
\newcommand{\bp}{\begin{proof}}
	\newcommand{\ep}{\end{proof}}
\newcommand{\bx}{\begin{example}}
	\newcommand{\ex}{\end{example}}
\newcommand{\br}{\begin{remark}}
	\newcommand{\er}{\end{remark}}
\newcommand{\be}{\begin{equation}}
	\newcommand{\ee}{\end{equation}}
\newcommand{\ba}{\begin{align}}
	\newcommand{\ea}{\end{align}}
\newcommand{\bn}{\begin{enumerate}}
	\newcommand{\en}{\end{enumerate}}
\newcommand{\bcs}{\begin{cases}}
	\newcommand{\ecs}{\end{cases}}

\newcommand{\s}{{\mathfrak{s}}}

\renewcommand{\AA}{\mathcal{A}}
\newcommand{\DD}{\mathcal{D}}
\newcommand{\EE}{\mathcal{E}}
\newcommand{\FF}{\mathcal{F}}
\newcommand{\FFF}{\mathrm{F}_{\Theta}}
\newcommand{\MM}{\mathcal{M}}
\newcommand{\NN}{\mathcal{N}}

\newcommand{\N}{\mathbb{N}}
\renewcommand{\SS}{\mathcal{S}}
\newcommand{\TT}{\mathcal{T}}

\newcommand{\XX}{\mathcal{X}}

\newcommand{\Cf}{\mathscr{C}}

\newcommand{\E}{\mathbb{E}}

\newcommand{\ov}{\overline}

\DeclareMathOperator{\torf}{\mathsf{torf_{\Theta}}}

\DeclareMathOperator{\tors}{\mathsf{tors}_{\Theta}}

\DeclareMathOperator{\ebrick}{\mathsf{ebrick}}
\DeclareMathOperator{\sbrick}{\mathsf{sbrick}}
\DeclareMathOperator{\mbrick}{\mathsf{mbrick}}
\DeclareMathOperator{\ccmbrick}{\mathsf{mbrick_{c.c.}}}

\DeclareMathOperator{\Fac}{\mathsf{Fac}}
\DeclareMathOperator{\0Filt}{\mathsf{Filt}}
\DeclareMathOperator{\aFilt}{\mathsf{Filt}_{\mathcal{A}}}

\DeclareMathOperator{\Filt}{\mathsf{Filt}_{\Theta}}

\DeclareMathOperator{\simp}{\mathsf{sim}}

\DeclareMathOperator{\Sub}{\mathsf{Sub}_{\Theta}}

\DeclareMathOperator{\cone}{\mathsf{cone}}
\DeclareMathOperator{\cocone}{\mathsf{cocone}}

\DeclareMathOperator{\lSchur}{\mathsf{Schur_L}}
\DeclareMathOperator{\rSchur}{\mathsf{Schur_R}}

\DeclareMathOperator{\add}{\mathsf{add}}

\makeatletter
\renewcommand{\section}{\@startsection{section}{1}{0mm}
	{-\baselineskip}{0.5\baselineskip}{\bf\leftline}}
\makeatother

\begin{document}
	
	\title[Monobricks in extriangulated length categories]{Monobricks in extriangulated length categories}
	\author[Y. Mei, L. Wang, J. Wei]{Yuxia Mei, Li Wang, Jiaqun Wei}
	\address{School of Mathematics-Physics and Finance, Anhui Polytechnic University, 241000 Wuhu, Anhui, P. R. China \endgraf}
	\email{meiyuxia2010@163.com {\rm(Y. Mei)}, wl04221995@163.com {\rm (L. Wang)}}
	\address{School of Mathematical Sciences, Zhejiang Normal University, 321004 Jinhua, Zhejiang, P. R. China\endgraf}
	\email{weijiaqun5479@zjnu.edu.cn {\rm (J. Wei)}}

	%\thanks{$\ast$: Corresponding author.}
	\subjclass[2020]{18E40; 18E10; 18G80.}
	\keywords{extriangulated length category; monobrick;  torsion-free class}
	
	\begin{abstract} In this paper, we introduce the notation of monobricks in an extriangulated length category as a generalization of the semibricks. We prove that there is a  bijection between monobricks and left  Schur subcategories. Then  we show that this bijection restricts to bijection between cofinally closed monobricks and torsion-free classes. These extend the results of Enomoto for abelian length categories.
	\end{abstract}
	
	\maketitle
	
	%%%%%%%%%%%%%%%%%%%%%%%%%%%%%%%%%%%%%%%%%%%%%%%%%%%%%%%%%%%%%%%%%%%%%%%%%%%%%%%%%%%%%%%%%%%%%%%%%%%%%%%%%%%%%%%%%%%%%%%%%%%%%%%%%%%%%%%%%%%%%%%%%%%%%%%%

	\section{Introduction}\label{sec:1}
An abelian length category  \cite{Ga} is an abelian category whose every object has a finite composition series. They are pivotal in representation theory and intimately connected to several important topics, e.g.  $\tau$-tilting theory \cite{En2}, torsion classes \cite{As} and the second Brauer–Thrall conjecture \cite{Kr}.

 Schur's lemma indicates that the isomorphism classes of simple objects in an abelian length category form a semibrick. The analogous concept in triangulated categories is a simple-minded system. It was first introduced in \cite{KL} to investigate its stable equivalences and the Auslander-Reiten conjecture. In  \cite{Du}, Dugas studied torsion pairs and  simple-minded systems in triangulated categories from the point of view of mutations. In particular, Dugas pointed that bounded derived categories of finite
 dimensional algebras with finite global dimension admit simple-minded systems. This allows us to regard these derived categories as analogs of abelian length categories in triangulated categories.

Recently, Nakaoka and Palu \cite{Na} introduced the notation of  extriangulated categories by
extracting properties on triangulated categories and exact categories. In \cite{WLZZ},  the authors introduced  extriangulated length categories as a generalization of abelian length categories and bounded derived categories of finite dimensional algebras with finite global dimension.  The idea behind in \cite{WLZZ} was to consider the filtration categories generated by simple-minded systems. In this setting, they introduced torsion(-free) classes and studied them from the perspective of lattice theory. They also provided a unified framework for studying the torsion classes and $\tau$-tilting theory. 

The notion of monobricks in an abelian length categories was first introduced by Enomoto \cite{Eno} to generalize the notion of semibricks. Using monobricks, the torsion(-free) classes can be classified using only the information on bricks. He also deduced several known
results on torsion(-free) classes and wide subcategories in length abelian categories without using $\tau$-tilting theory.

 These works inspire us to develop a general framework for the classification of torsion-free classes in extriangulated length categories via monobricks. For this purpose, we investigate the monobricks and left  Schur subcategories in an extriangulated length category. We prove that there is a  bijection between monobricks and left  Schur subcategories  (see Theorem \ref{thm:schur-mbrick}). Subsequently, we show that this bijection restricts to bijection between cofinally closed monobricks and torsion-free classes. Then we have the following  main results of this paper.

	\begin{theoremi}\label{AA} {\rm (see Theorem \ref{thm:torfproj}~~for details)} \label{thm:B}
		Let $(\AA, \Theta)$ be an extriangulated length category. Then we have the following commutative diagram and the horizontal maps are bijections.
		\[
	\begin{tikzcd}
		\torf\AA \rar["\simp", shift left] \ar[dd, bend right=75, "1"'] \dar[hookrightarrow] & \ccmbrick\AA \lar["\Filt", shift left]\dar[hookrightarrow] \ar[dd, bend left=75, "1"]\\
		\lSchur\AA \rar["\simp", shift left] \dar["\FFF", twoheadrightarrow] & \mbrick\AA \dar["(\ov{?}) ", twoheadrightarrow] \lar["\Filt", shift left]  \\
		\torf\AA \rar["\simp", shift left] & \ccmbrick\AA \lar["\Filt", shift left]
	\end{tikzcd}
	\]
	\end{theoremi}

	\begin{org}
		This paper is organized as follows.
		In Section \ref{sec:2}, we summarize some basic definitions and properties of extriangulated (length) categories.
		In Section \ref{sec:3}, we introduces the notions of left Schur subcategories and monobricks in extriangulated length categories. Then we establish a   bijection between them.
		In Section \ref{sec:4}, we
		apply our results on torsion-free classes and then prove Theorem A.

	\end{org}
	
	\begin{conv} Throughout this paper, we assume that all considered categories are skeletally small and Krull-Schmidt, and that the subcategories are full and closed under isomorphisms. For an additive category $\Cf$, we denote by $\Iso(\Cf)$ the set of isomorphism class of objects in $\Cf$.  We often identify an isomorphism class in a category with its representative. For a collection $\MM$ of objects in $\Cf$, we denote by $\add\MM $ the subcategory of $\Cf$ consisting of direct summands of finite direct sums of objects in $\MM$.
	\end{conv}

	\section{Preliminaries}\label{sec:2}
	
	In this section, we review some basic definitions and properties of extriangulated (length) categories from \cite{Na} and \cite{WLZZ}.

	\subsection{Extriangulated categories}
	Let $\AA$ be an additive category and let $\E$: $\AA^{\rm op}\times\AA\rightarrow Ab$ be a biadditive functor. For any pair of objects $A$, $C\in\AA$, an element $\delta\in \E(C,A)$ is called an {\em $\E$-extension}. The zero element $0\in\E(C,A)$ is called the {\em split $\E$-extension}.
	For any morphism $a\in \Hom_{\AA}(A,A')$ and $c\in \Hom_{\AA}(C',C)$, we have
	$$\E(C,a)(\delta)\in\E(C,A')~\text{and}~\E(c,A)(\delta)\in\E(C',A).$$
	We simply denote them by $a_{\ast}\delta$ and $c^{\ast}\delta$,~respectively.  A morphism $(a,c)$: $\delta\rightarrow\delta'$ of $\E$-extensions is a pair of morphisms $a\in  \Hom_{\AA}(A,A')$ and $c\in  \Hom_{\AA}(C,C')$ satisfying the equality $a_{\ast}\delta=c^{\ast}\delta'$. By the biadditivity of $\mathbb{E}$, we have a natural isomorphism
	$$\mathbb{E}(C\oplus C',A\oplus A')\cong\mathbb{E}(C,A)\oplus\mathbb{E}(C,A')\oplus\mathbb{E}(C',A)\oplus\mathbb{E}(C',A').$$
	Let $\delta\oplus \delta'\in\mathbb{E}(C\oplus C',A\oplus A')$ be the element corresponding to $(\delta,0,0,\delta')$ through
	this isomorphism. Two sequences of morphisms $A\stackrel{x}{\longrightarrow}B\stackrel{y}{\longrightarrow}C$ and $A\stackrel{x'}{\longrightarrow}B'\stackrel{y'}{\longrightarrow}C$ in $\AA$ are said to be {\em equivalent} if there exists an isomorphism $b\in  \Hom_{\AA}(B,B')$ such that the following diagram
	$$\xymatrix{
		A \ar@{=}[d] \ar[r]^-{x} & B\ar[d]_{b}^-{\simeq} \ar[r]^-{y} & C\ar@{=}[d] \\
		A \ar[r]^-{x'} &B' \ar[r]^-{y'} &C  }$$ is commutative.
	We denote the equivalence class of $A\stackrel{x}{\longrightarrow}B\stackrel{y}{\longrightarrow}C$ by $[A\stackrel{x}{\longrightarrow}B\stackrel{y}{\longrightarrow}C]$. In addition, for any $A,C\in\AA$, we denote as
	$$0=[A\stackrel{\tiny\begin{pmatrix} 1\\0\end{pmatrix}}{\longrightarrow}A\oplus C\stackrel{(0~1)}{\longrightarrow}C].$$
	For any two classes $[A\stackrel{x}{\longrightarrow}B\stackrel{y}{\longrightarrow}C]$ and $[A'\stackrel{x'}{\longrightarrow}B'\stackrel{y'}{\longrightarrow}C']$, we denote as
	$$[A\stackrel{x}{\longrightarrow}B\stackrel{y}{\longrightarrow}C]\oplus[A'\stackrel{x'}{\longrightarrow}B'\stackrel{y'}{\longrightarrow}C']=
	[A\oplus A'\stackrel{x\oplus x'}{\longrightarrow}B\oplus B'\stackrel{y\oplus y'}{\longrightarrow}C\oplus C'].$$
	
	\begin{definition}
		Let $\s$ be a correspondence which associates an equivalence class $\s(\delta)=[A\stackrel{x}{\longrightarrow}B\stackrel{y}{\longrightarrow}C]$ to any $\E$-extension $\delta\in\E(C,A)$. We say $\s$ is  a {\em realization} of $\E$ if for any morphism $(a,c):\delta\rightarrow\delta'$ with 
		$\s(\delta)=[A\stackrel{x}{\longrightarrow}B\stackrel{y}{\longrightarrow}C]$ and $\s(\delta')=[A'\stackrel{x'}{\longrightarrow}B'\stackrel{y'}{\longrightarrow}C']$, there exists a morphism $b\in\Hom_{\AA}(B,B')$ such that the following diagram commutative.
		$$\xymatrix{
			A \ar[d]_-{a} \ar[r]^-{x} & B  \ar[r]^{y}\ar[d]_-{b} & C \ar[d]_-{c}    \\
			A'\ar[r]^-{x'} & B' \ar[r]^-{y'} & C' .   }		$$  
		A realization $\s$ of $\E$ is said to be {\em additive} if it satisfies the following conditions:
		
		(a) For any $A,~C\in\AA$, the split $\E$-extension $0\in\E(C,A)$ satisfies $\s(0)=0$.
		
		(b) $\s(\delta\oplus\delta')=\s(\delta)\oplus\s(\delta')$ for any pair of $\E$-extensions $\delta$ and $\delta'$.
	\end{definition}

	\begin{definition} (\cite[Definition 2.12]{Na})\label{F}
		We call the triplet $(\AA, \E,\s)$ an {\em extriangulated category} if it satisfies the following conditions:\\
		$\rm(ET1)$ $\E$: $\AA^{op}\times\AA\rightarrow Ab$ is a biadditive functor.\\
		$\rm(ET2)$ $\s$ is an additive realization of $\E$.\\
		$\rm(ET3)$ Let $\delta\in\E(C,A)$ and $\delta'\in\E(C',A')$ be any pair of $\E$-extensions, realized as
		$\s(\delta)=[A\stackrel{x}{\longrightarrow}B\stackrel{y}{\longrightarrow}C]$, $\s(\delta')=[A'\stackrel{x'}{\longrightarrow}B'\stackrel{y'}{\longrightarrow}C']$. For any commutative square in $\AA$
		$$\xymatrix{
			A \ar[d]_{a} \ar[r]^{x} & B \ar[d]_{b} \ar[r]^{y} & C \\
			A'\ar[r]^{x'} &B'\ar[r]^{y'} & C'}$$
		there exists a morphism $(a,c)$: $\delta\rightarrow\delta'$ which is realized by $(a,b,c)$.\\
		$\rm(ET3)^{op}$~Dual of $\rm(ET3)$.\\
		$\rm(ET4)$~Let $\delta\in\E(D,A)$ and $\delta'\in\E(F,B)$ be $\E$-extensions realized by
		$A\stackrel{f}{\longrightarrow}B\stackrel{f'}{\longrightarrow}D$ and $B\stackrel{g}{\longrightarrow}C\stackrel{g'}{\longrightarrow}F$, respectively.
		Then there exist an object $E\in\AA$, a commutative diagram
		\begin{equation}\label{2.1}
			\xymatrix{
				A \ar@{=}[d]\ar[r]^-{f} &B\ar[d]_-{g} \ar[r]^-{f'} & D\ar[d]^-{d} \\
				A \ar[r]^-{h} & C\ar[d]_-{g'} \ar[r]^-{h'} & E\ar[d]^-{e} \\
				& F\ar@{=}[r] & F   }
		\end{equation}
		in $\AA$, and an $\E$-extension $\delta''\in \E(E,A)$ realized by $A\stackrel{h}{\longrightarrow}C\stackrel{h'}{\longrightarrow}E$, which satisfy the following compatibilities:\\
		$(\textrm{i})$ $D\stackrel{d}{\longrightarrow}E\stackrel{e}{\longrightarrow}F$ realizes $\E(F,f')(\delta')$,\\
		$(\textrm{ii})$ $\E(d,A)(\delta'')=\delta$,\\
		$(\textrm{iii})$ $\E(E,f)(\delta'')=\E(e,B)(\delta')$.\\
		$\rm(ET4)^{op}$ Dual of $\rm(ET4)$.
	\end{definition}

In this section, we always assume that 	$\AA:= (\AA,\E,\s)$ is an extriangulated category.

	\begin{example}\label{E-2.3}

	$(1)$ Let $\AA$ be an exact category. For any $A,C\in\AA$, we define $\mathbb{E}(C, A)$ to be the collection of all equivalence classes of short exact sequences of the form  $A\stackrel{}{\longrightarrow}B\stackrel{}{\longrightarrow}C$. For any $\delta\in\mathbb{E}(C, A)$, the realization $\s(\delta)$ is defined as $\delta$ itself. Then $(\AA, \mathbb{E}, \s)$ is an extriangulated category. We refer the reader to \cite[Example 2.13]{Na} for more details.
	
	$(2)$ Let $\TT$ be a triangulated category with suspension functor $\Sigma$. For any $A,C\in\TT$, we define $\mathbb{E}(C, A):=\Hom_{\TT}(C,\Sigma A)$. For any $\delta\in\mathbb{E}(C, A)$,  take a distinguished triangle
	\[
	A\stackrel{}{\longrightarrow}B\stackrel{}{\longrightarrow}C\stackrel{\delta}{\longrightarrow} \Sigma A
	\]
	and define  $\s(\delta) = [A\stackrel{x}{\longrightarrow}B\stackrel{y}{\longrightarrow}C]$. Then $(\TT, \mathbb{E}, \s)$ is an extriangulated category.  We refer the reader to \cite[Example 2.13]{Na} for more details.
	
	$(3)$   Let $\TT$ be an extension-closed subcategory of $\AA$. We define $\E_\TT$ to be the restriction of
	$\E$ onto $\TT^{\rm op}\times \TT$ and define $\E_\TT$ by restricting $\s$ on $\E_\TT$. Then ($\TT,\E_\TT,\s_\TT$) is 
	an extriangulated category  (cf. \cite[Remark 2.18]{Na}).	
\end{example}
	
 Given $\delta\in\E(C,A)$. If $\s(\delta)=[A\stackrel{x}{\longrightarrow}B\stackrel{y}{\longrightarrow}C]$, then the sequence $A\stackrel{x}{\longrightarrow}B\stackrel{y}{\longrightarrow}C$ is called a {\em conflation}, $x$ is called an {\em inflation} and $y$ is called a {\em deflation}. In this case, we call $A\stackrel{x}{\longrightarrow}B\stackrel{y}{\longrightarrow}C\stackrel{\delta}\dashrightarrow$ is an $\E$-triangle.
	We will write $A=\cocone(y)$ and $C=\cone(x)$ if necessary. Let $\TT,\mathcal{F}$ be two subcategories of  $\AA$. We define
	$$\TT^{\perp}=\{M\in\AA~|~\Hom_{\AA}(X,M)=0~\text{for any}~X\in\TT\},$$
	$$^{\perp}\TT=\{M\in\AA~|~\Hom_{\AA}(M,X)=0~\text{for any}~X\in\TT\},$$
	$$\TT\ast \FF=\{M\in\AA \mid \text{there exists an } \E\text{-triangle }  T\stackrel{}{\longrightarrow}M\stackrel{}{\longrightarrow}F\stackrel{}\dashrightarrow \text{with}~T\in\TT,F\in\FF\}.$$
	We say that a subcategory $\TT \subseteq \AA$ is {\em extension-closed} if $\TT\ast \TT \subseteq \TT$.

	\subsection{Extriangulated length categories}
	We begin with review the definition of the length categories and some related results from  \cite{WLZZ}.
	
	\begin{definition}\label{D-1} (\cite[Definition 3.1]{WLZZ}) We say that a map $\Theta:\Iso(\AA)\rightarrow \N$ is a {\em length function} on $\AA$ if it satisfies the following conditions:
		\begin{enumerate} 
			\item $\Theta(X)=0$ if and only if $X\cong0$.
			\item For any $\E$-triangle $X\stackrel{}{\longrightarrow}L\stackrel{}{\longrightarrow}M\stackrel{\delta}\dashrightarrow$ in $\AA$, we have $\Theta(L)\leq \Theta(X)+\Theta(M)$. In addition, if $\delta=0$, then $\Theta(L)=\Theta(X)+\Theta(M)$.
		\end{enumerate} 
		A length function $\Theta$ is said to {\em stable} if
		it satisfies the condition $\Theta(L)=\Theta(X)+\Theta(M)$ for any $\E$-triangle $X\stackrel{}{\longrightarrow}L\stackrel{}{\longrightarrow}M\stackrel{}\dashrightarrow$ in $\AA$.
	\end{definition}
	
	Let $\Theta$ be a length function on $\AA$. We say  an $\E$-triangle $X\stackrel{x}{\longrightarrow}L\stackrel{y}{\longrightarrow}M\stackrel{\delta}\dashrightarrow$ in $\AA$ is {\em $\Theta$-stable} (or {\em stable} when $\Theta$ is clear) if $\Theta(L) = \Theta(X) +\Theta(M)$. In this case, $x$ is called a  {\em $\Theta$-inflation}, $y$ is called a   {\em $\Theta$-deflation} and $\delta$ is called a  $\Theta$-extension.  In analogy with short exact sequences, we depict $\Theta$-inflation by  $\rightarrowtail $ and  $\Theta$-deflation by $\twoheadrightarrow $.  A morphism $f:M\rightarrow N$ in $\AA$ is called {\em $\Theta$-admissible} if $f$ admits a
	{\em $\Theta$-decomposition} $(i_{f}, X_{f},j_{f})$, i.e. there is a commutative diagram
	\begin{equation*}
		\xymatrix{
			M\ar[rr]^{f}\ar@{->>}[dr]_-{i_{f}} & &N \\
			&   X_{f} \ar@{>->}[ur]_{j_{f}}   &      }
	\end{equation*}
	such that $i_{f}:M\twoheadrightarrow X_{f}$ is a $\Theta$-deflation and $j_{f}:X_{f}\rightarrowtail N$ is a $\Theta$-inflation. The following observation is simple but useful.

	\begin{remark}\label{R-2-4} Let  $\Theta$  be a  length function on $\AA$. Take a $\Theta$-stable $\E$-triangle 	$X\stackrel{f}\rightarrowtail L\stackrel{g}\twoheadrightarrow M\stackrel{}\dashrightarrow$. It is easily check that $\Theta(X)=\Theta(L)$ if and only if $f$ is an isomorphism. Similarly, we have $\Theta(L)=\Theta(M)$ if and only if $g$ is an isomorphism. 
	\end{remark}

	\begin{definition}  \cite[Definition 3.8]{WLZZ} Let  $\Theta$  be a  length function on $\AA$. We say that $((\AA, \E,\s),\Theta)$ is an  {\em  extriangulated length category},  or simply a {\em length category}, if every morphism in $\AA$ is $\Theta$-admissible.
		For simplicity, we often write  $(\AA,\Theta)$ for  $((\AA, \E,\s),\Theta)$ when $\E$ and $\s$ are clear from the context.
	\end{definition}
	
	Let $\XX$ be a collection of objects in $\AA$. The {\em filtration subcategory} $\aFilt\XX$ is consisting of all objects $M$ admitting a finite filtration of the form
	\begin{equation*}
		0=M_{0}\stackrel{f_{0}}{\longrightarrow}M_{1}\stackrel{f_{1}}{\longrightarrow}M_{2}{\longrightarrow}\cdots\xrightarrow{f_{n-1}}M_{n}=M
	\end{equation*}
	with $f_{i}$ being an inflation and $\cone(f_{i})\in\XX$ for any $0\leq i\leq n-1$.  For each object $M\in\aFilt\XX$, the minimal length of $\XX$-{filtrations} of $M$ is called the $\XX$-{\em length} of $M$, which is denoted by $l_{\XX}(M)$. It was proven in \cite[Lemma 2.8]{Wa2} that $\aFilt\XX$ is the smallest extension-closed subcategory in $\AA$ containing $\XX$.
	
	An object $M\in\AA$ is called a {\em brick} if its endomorphism ring is a division ring. Let  $\XX$ be a set of isomorphism classes of bricks in $\AA$. We say $\XX$ is a {\em semibrick} if $\Hom_{\AA}(X_1,X_2)=0$ for any two non-isomorphic objects $X_1,X_2$ in $\XX$. If moreover $\AA=\aFilt\XX$, then we say $\XX$ is a {\em simple-minded system} in $\AA$. 
	A semibrick $\XX$ is said to be {\em proper} if for any $X\in\XX$,
	there does not exist an $\E$-triangle of the form $X\stackrel{}{\longrightarrow}0\stackrel{}{\longrightarrow}N\stackrel{}\dashrightarrow$ with $N\in\aFilt\XX$.
	
	Let $(\AA,\Theta)$ be a length category. Set
	$$\Theta_{1}:=\{0\neq M\in\Iso(\AA)~|~0<\Theta(M)\leq \Theta(N)~\text{for any}~0\neq N\in\Iso(\AA)\}.$$
	Without loss of generality, we may assume that $\Theta(M)=1$ for any $M\in\Theta_{1}$. For $n\geq2$, we  inductively define various sets  as follows:
	$$\Theta'_{n}=\{M\in \AA~|~M\in\Theta_{n-1}^{\perp}\bigcap{^{\perp}}\Theta_{n-1},\Theta(M)=n\}~\text{and}~\Theta_{n}=\Theta_{n-1}\bigcup\Theta'_{n}.$$
	Set $\Theta_{\infty}=\bigcup_{n\geq1}\Theta_{n}$. We have the following useful characterization of length categories by using semibricks.
	
	\begin{theorem}\label{main0}  Let $\AA$ be an extriangulated category.
		\begin{enumerate} 
			\item If $(\AA,\Theta)$ is a length category, then $\Theta_{\infty}$ is a simple-minded system in $\AA$.
			\item If $\XX$ is a simple-minded system in $\AA$, then  $(\AA,l_{\XX})$ is a length category. In addition, $l_{\XX}$ is stable if and only if $\XX$ is proper.	
		\end{enumerate} 
	\end{theorem}
	\begin{proof} This follows from \cite[Lemma 3.3, Theorem 3.9, Proposition 3.13]{WLZZ}.
	\end{proof}
	
	\begin{example}\label{E-2.8} $(1)$ Let $\AA$ be an abelian length category. Observe that the set ${\rm sim}(\AA)$ consisting of the isomorphism classes of simple objects in $\AA$ is a simple-minded system. Then Theorem \ref{main0} implies that $(\AA,l_{{\rm sim}(\AA)})$ is a stable length category.
		
		$(2)$ Let $\Lambda$ be a finite dimensional algebra of finite global dimension. By \cite[Example 2]{Du}, there exists a simple-minded system $\XX$ in bounded derived category $D^{b}(\Lambda)$. Then $(D^{b}(\Lambda),l_{\XX})$ is a length category. We refer the reader to \cite[Example 3.25]{WLZZ} for more details.
	\end{example}
	
	We collect some results on length categories which we will need.

	\begin{proposition}\label{pro-length-cat} Let $(\AA,\Theta)$ be a length category. 
		\begin{enumerate} 
			\item The classes of $\Theta$-inflations (resp. $\Theta$-deflations) is closed under compositions.
			\item Let $f:X\rightarrow Y$ and  $g:Y\rightarrow Z$ be any composable pair of morphisms in $\AA$. If $gf$ is a $\Theta$-inflation, then so is $f$. Dually, if $gf$ is a $\Theta$-deflation, then so is $g$.
		\end{enumerate} 
	\end{proposition}
	\begin{proof}The reader can find the statement (1) in \cite[Lemma 3.6]{WLZZ} and (2) in \cite[Lemma 3.20]{WLZZ}. 
	\end{proof}

	\begin{remark}\label{filt-0}
		Let $((\AA, \E,\s),\Theta)$ be a length category. In general, the length function $\Theta$ is not stable. For any $A,B\in \AA$, we define $\E_{\Theta}(A,B)$ to be the subset of $\E(A,B)$  consisting of all $\Theta$-extensions. We define $\s_{\Theta}$ as the restriction of $\s$ on $\E_{\Theta}$. It was proven in \cite[Proposition 3.22]{WLZZ} that $((\AA, \E_{\Theta},\s_{\Theta}),\Theta)$ is a stable length category. 	
		
		Let $\XX$ be a collection of objects in $\AA$. We denote by $\Filt\XX$ the filtration subcategory generated by $\XX$ in $(\AA, \E_{\Theta},\s_{\Theta})$. It is obvious that  $\Filt\XX\subseteq\aFilt\XX$. Note that $\Theta$ is stable in $(\AA, \E_{\Theta},\s_{\Theta})$. By \cite[Proposition 2.2(3)]{WLZZ}, any object $M\in\Filt\XX$ have  two stable $\E$-triangles
		$$X_{1}\stackrel{}\rightarrowtail M\stackrel{}\twoheadrightarrow M'\stackrel{}\dashrightarrow~\text{and}~~M''\stackrel{}\rightarrowtail M\stackrel{}\twoheadrightarrow X_2\stackrel{}\dashrightarrow$$
		such that $X_{1},X_{2}\in \mathcal{X}$. For two subcategories $\TT,\mathcal{F}\subseteq\AA$, we define
		$$\TT\ast_{\Theta} \FF=\{M\in\AA \mid \text{there exists an } \E\text{-triangle}~  T\stackrel{}{\rightarrowtail} M\stackrel{}{\twoheadrightarrow} F\stackrel{}\dashrightarrow \text{with}~ T\in\TT,F\in\FF\}.$$
		It is obvious that $\TT\ast_{\Theta} \FF\subseteq \TT\ast \FF$. We say a subcategory $\TT$ of $\AA$ is {\em closed under $\Theta$-extensions} if $\TT\ast_{\Theta} \TT \subseteq \TT$. Then one can check that $\Filt\XX$ is closed under $\Theta$-extensions.
	\end{remark}

	\section{Monobricks  and left Schur subcategories}\label{sec:3}
	
	In what follows, we always assume that 	$(\AA, \Theta):= ((\AA,\E,\s), \Theta)$ is a length category. The aim of this section is to investigate the monobricks and  left Schur subcategories in length categories. We start with introducing the following notations.

	\begin{definition}\label{monobrick}
		Let $\MM$ be a set of isomorphism classes of bricks in $\AA$. We say $\MM$ is a {\em monobrick} if every morphism between elements of $\MM$ is either zero or a $\Theta$-inflation in $\AA$.
	\end{definition}
	
	We denote by $\mbrick \AA$ (resp. $\sbrick \AA$) the set of monobricks (resp. semibricks) in $\AA$. Then clearly we have  $\sbrick \AA \subseteq \mbrick \AA$. 
	
	\begin{remark} (1) Let $\AA$ be an abelian length category. Recall that $(\AA,l_{{\rm sim}(\AA)})$ is a stable length category (see Example \ref{E-2.8}). In this case,  our definition coincides with that in \cite[Definition 2.1(1)]{Eno}.
		
		$(2)$ It is natural to define an {\em epibrick} in length categories. That is, a set $\MM$ of isomorphism classes of bricks satisfies every morphism between elements of $\MM$ is either zero or a $\Theta$-deflation in $\AA$. 	We denote by $\ebrick \AA$ the set of epibricks in $\AA$.
	\end{remark}

	Let $\mathcal{W}$ be a subcategory of $\AA$. A non-zero object $M\in\mathcal{W}$ is called a {\em simple object}  (c.f. \cite[Section 3.3]{WLZZ}) if there does not exist an $\E$-triangle
	$A\stackrel{}\rightarrowtail M\stackrel{}\twoheadrightarrow B\stackrel{}\dashrightarrow$ in $\mathcal{W}$ such that $A,B\neq 0$. We denote by $\simp(\mathcal{W})$ the set of isomorphism classes of simple objects in $\mathcal{W}$.  
	\begin{definition}\label{left-schurian} Let $\EE$ be a subcategory of $\AA$.
		\begin{enumerate} 		
			\item We say a non-zero object $M \in \AA$ is {\em left Schurian for $\EE$} if every morphism $M \rightarrow C$ with $C \in \EE$ is either zero or a $\Theta$-inflation.
			\item We say $\EE$ is a {\em left Schur subcategory} of $\AA$ if it satisfies the following conditions:
	 	  \begin{enumerate} 
		  	  \item $\EE$ is closed under  $\Theta$-extensions.
			  \item If $M\in\simp(\EE)$, then $M$ is left Schurian for $\EE$.
		  \end{enumerate} 
		\end{enumerate} 
	\end{definition}
The notion of  {\em right Schurian subcategories} is defined dually. 	We denote by $\lSchur\AA$ (resp. $\rSchur\AA$) the set of left (resp. right) Schur subcategories of $\AA$.

	\begin{lemma}\label{schurian-simple}
		Let $\EE\in\lSchur\AA$.
		\begin{enumerate}
			\item The set $\simp(\EE)$ is a monobrick.
			\item If $M\in\EE$ is left Schurian for $\EE$, then $M \in\simp(\EE)$.
		\end{enumerate}
	\end{lemma}
	\begin{proof} (1)   Take a non-zero morphism $f:M \rightarrow N$ with $M,N\in\simp(\EE)$. Then Definition \ref{left-schurian} implies that $f$ is a $\Theta$-inflation. 
		
		(2) Let
		$L\stackrel{}\rightarrowtail M\stackrel{f}\twoheadrightarrow N\stackrel{}\dashrightarrow$ be an $\E$-triangle in $\EE$. By hypothesis, either $f=0$ or $f$ is a $\Theta$-inflation. In the former case, we have $N\cong 0$ by Remark \ref{R-2-4}. In the latter, we infer that $f$ is an isomorphism and thus $L\cong 0$. This implies that $M \in\simp(\EE)$.
	\end{proof}

	\begin{lemma}\label{lem-schur1} Let $\mathcal{C}$ and $\DD$ be two subcategories of $\AA$ and let $M\in\AA$. If $M$ is left Schurian for $\mathcal{C}$ and $\DD$, then $M$ is left Schurian for $\mathcal{C}\ast_{\Theta}\DD$.
	\end{lemma}
	\begin{proof}
		Let $C\stackrel{f}{\rightarrowtail} X\stackrel{g}{\twoheadrightarrow} D\stackrel{}\dashrightarrow$ be  an $\E$-triangle with $C \in \mathcal{C}$ and $D \in \DD$. For any morphism $\varphi:M \rightarrow X$, we consider the following diagram
		\begin{equation*}
			\xymatrix{
				&   M\ar[d]^{\varphi}& & \\
				C\ar@{>->}[r]^{f}  & X\ar@{->>}[r]^{g} & D  \ar@{-->}[r]^{ } & .}
		\end{equation*}
		Since $M$ is left Schurian for $\DD$, the morphism $g\varphi:M\rightarrow D$ is either zero or a $\Theta$-inflation. For the latter, the morphism $\varphi$ is  a $\Theta$-inflation  by Proposition \ref {pro-length-cat}(2). For the former, there exists a morphism $\overline{\varphi}:M \rightarrow C$ such that $\varphi=f\overline{\varphi}$. Note that $M$ is a left Schurian for $\mathcal{C}$. Thus $\overline{\varphi}$ is either zero or a $\Theta$-inflation. In the former case, we have $\varphi=0$, and in the latter, we know that $\varphi$ is a $\Theta$-inflation by Proposition  \ref {pro-length-cat}(1).
	\end{proof}
	
	\begin{lemma}\label{lem-schur2} Let $\MM$ be a monobrick in $\AA$ and  take                                                                                                                                                                                                                                                                                                                                                                                                                                                                                                                                                                                                                                                                                                                                                                                                                                                                                                                                                                                                                                                                                                                                                                                                                                                                                                                                                                                                                                                                                                                                                                                                                                                                                                                                                                                                                                                                                                                                                                                                                                                                                                                     $M\in \Filt \MM$. Then the following statements are equivalent:
		\begin{enumerate} 
			\item $M\in\simp(\Filt \MM)$.
			\item $M\in\MM$.
			\item $M$ is left Schurian for $\Filt \MM$.
		\end{enumerate} 
	\end{lemma}
	\begin{proof}	
		(1)$\Rightarrow$(2) By Remark \ref{filt-0}, there exists an $\E$-triangle  $X\stackrel{}\rightarrowtail M\stackrel{}\twoheadrightarrow N\stackrel{}\dashrightarrow$ in   $\Filt \MM$ such that $X\in \MM$. Since $M\in\simp(\Filt \MM)$, we infer that $M\cong X\in\MM$.
		
		(2)$\Rightarrow$(3) Note that $M$ is left Schurian for $\MM$. By using Lemma \ref{lem-schur1} repeatedly, we conclude that $M$ is left Schurian for $\Filt \MM$.
		
		(3)$\Rightarrow$(1) This follows from Lemma \ref{schurian-simple}(2).
	\end{proof}

	We can now state the relationship between monobricks and left Schur subcategories in length categories. These results was proved by Enomoto in \cite[Theorem 2.11]{Eno} for abelian length categories.

	\begin{theorem}\label{thm:schur-mbrick}
		Let $((\AA,\E,\s), \Theta)$ be an extriangulated length category.
		There exists a bijection between
		\[
		\begin{tikzcd}[row sep = small]
			\lSchur\AA \rar["\simp", shift left] & \mbrick\AA \lar["\Filt", shift left] .
		\end{tikzcd}
		\]
	\end{theorem}
	\begin{proof} Let $\EE$ be a left Schur subcategory of $\AA$. Then Lemma \ref{schurian-simple}(1) implies that $\simp(\EE)$ is a monobrick. Note that $\EE$ is closed under $\Theta$-extensions. Thus $\Filt (\simp(\EE )) \subseteq\EE$. On the other hand, we take  $M\in\EE$ and set $t=\min\{\Theta(M)~|~M \in \EE \}$.  We claim that $M\in\Filt (\simp(\EE ))$. If $\Theta(M) = t$, then we clearly have $M \in \simp(\EE)$. For $\Theta(M)>t$, we may assume that  $M \notin \simp(\EE)$. Then there exists an $\E$-triangle $X\stackrel{}{\rightarrowtail} M\stackrel{}{\twoheadrightarrow} Y\stackrel{}\dashrightarrow$ in $\EE$ such that $\max\{\Theta(X), \Theta(Y)\} < \Theta (M)$.  By induction hypothesis, we have $X, Y \in \Filt (\simp(\EE )) $ and so is $M$. The observation above implies that $\Filt (\simp(\EE ))=\EE$. Conversely, we take a monobrick $\MM$ in $\AA$. Recall that $\Filt \MM$ is closed under $\Theta$-extensions. By using Lemma \ref{lem-schur2}, we infer that $\Filt \MM$ is a left Schur subcategory. Again by Lemma \ref{lem-schur2}, we have $\EE=\simp(\Filt\MM)$. This proves $\simp$ and $\Filt$ are mutually inverse isomorphisms.
	\end{proof}
	
\begin{remark}
Using the analogous arguments as those proving Theorem \ref{thm:schur-mbrick}, we can prove there exists a bijection between 	$\rSchur\AA$ and $\ebrick \AA$.
\end{remark}

	\section{Torsion-free classes and cofinally closed monobricks}\label{sec:4}
   The aim of this section is to investigate torsion-free classes in length categories by using Theorem \ref{thm:schur-mbrick}. It turns out that torsion-free classes are precisely left Schurian subcategories generated by cofinally closed monobricks. 

  We begin with recalling the following  definitions from \cite[Section 4]{WLZZ}. For a subcategory $\SS\subseteq\AA$, we define
 	 $$\Fac_{\Theta}(\SS)=\{M\in \AA~|~\text{there exists a $\Theta$-deflation}~S\twoheadrightarrow M~\text{for some}~S\in \SS\},$$
	 $$\Sub(\SS)=\{M\in \AA~|~\text{there exists a $\Theta$-inflation}~M\rightarrowtail S~\text{for some}~S\in \SS\}.$$
	
	\begin{definition}\label{Torsion}\cite[Definition 4.1]{WLZZ}
		A {\em torsion class}  in $(\AA,\Theta)$ is a subcategory $\TT$ of $\AA$ such that $\Fac_{\Theta}(\TT)\subseteq\TT$ and $\TT\ast_{\Theta}\TT\subseteq\TT$. The set of torsion classes in $(\AA,\Theta)$ is denoted by $\tors \AA$. Dually, a subcategory $\FF$ of $\AA$ is called a {\em torsion-free class} in $(\AA,\Theta)$ if $\Sub(\FF)\subseteq\FF$ and $\FF\ast_{\Theta}\FF\subseteq\FF$. The set of torsion-free classes in $(\AA,\Theta)$ is denoted by $\torf \AA$.
	\end{definition}
	
    \begin{remark} We note that  torsion-classes  (resp. torsion-free classes) are extension-closed (cf. \cite[Lemma 4.3]{WLZZ}). On the other hand, there exist mutually inverse bijections
%  $$
%   \xymatrix@C=3.5pc{ \tors \AA\ar@<-1ex>[r]_-{(-)^{\perp}}&
%	\torf \AA\ar@<-1ex>[l]_-{{^{\perp}}(-)}.        }
%  $$	
	\[
         \begin{tikzcd}[row sep = small, scale = 1.5 ]
  	         \tors \AA \rar["(-)^{\perp}", shift left] & \torf \AA \lar["{^{\perp}}(-)", shift left] .
         \end{tikzcd}
     \]
      We refer the reader to \cite[Section 4]{WLZZ} for more details.
     \end{remark}

	The observations following show that torsion-free classes are left Schur subcategories.
	
	\begin{lemma}\label{prop:torf-schur}
		We have $\torf \AA\subseteq\lSchur\AA$.
	\end{lemma}
	\begin{proof}  Let $\mathcal{F}$ be a torsion-free class in $(\AA,\Theta)$. Take a non-zero morphism $f:M\rightarrow  X$ in $\mathcal{F}$ with $M\in\simp(\FF)$. It is suffices to show that $f$ is a $\Theta$-inflation. To see this, we take a $\Theta$-decomposition $(i_f,X_f,j_f)$ of $f$. Since $\FF$ is a torsion-free class, we have $\cocone(i_f),X_f\in\FF$. Since $M\in\simp(\FF)$, we infer that $i_f$ is an isomorphism and hence $f\cong j_f$ is a $\Theta$-inflation.
	\end{proof}

	For a subcategory $\SS\subseteq \AA$, we define
	$$\mathrm{T}_{\Theta}(\SS):=\bigcap_{\TT\in \tors \AA;\SS\subseteq\mathcal{ T}}\TT~\text{and}~\mathrm{F}_{\Theta}(\SS):=\bigcap_{\mathcal{F}\in\torf \AA;\SS\subseteq\mathcal{F}}\mathcal{F}.$$
	Clearly, $\mathrm{T}_{\Theta}(\SS)$ and $\mathrm{F}_{\Theta}(\SS)$ are the smallest torsion class and the smallest torsion-free class containing $\SS$, respectively. The following result is quiet useful.
	
	\begin{proposition}\label{P-3-7} {\rm (\cite[Proposition 4.8]{WLZZ})} Let $\SS$ be a subcategory of $\AA$. Then 
		$$\mathrm{T}_{\Theta}(\SS)=\Filt(\Fac_{\Theta}(\SS))~{\text and}~\mathrm{F}_{\Theta}(\SS)=\Filt(\Sub(\SS)).$$
	\end{proposition}
%As a consequence, we have the following observation.
%\begin{lemma} \label{Filt-cor}
%Let $\SS$ be a subcategory of $\AA$. Then $\FFF(\SS)=\FFF(\Filt(\SS))$.
%\end{lemma}
	
%\begin{proof} The inclusion $\FFF(\SS)\subseteq\FFF(\Filt(\SS))$ is obvious. By Proposition \ref{P-3-7}, we have $\Filt(\SS) \subseteq \Filt(\Sub(\SS)) = \FFF(\SS)$. Recall that $\FFF(\Filt(\SS))$ is the smallest torsion-free class containing $\Filt(\SS)$. It turns out that $\FFF(\Filt(\SS))\subseteq \FFF(\SS).$
%\end{proof}

	\begin{definition}\label{def:cof}
		Let $\MM$ and $\NN$ be two monobricks in $\AA$. We say $\NN$ is a {\em cofinal extension of $\MM$}, or $\MM$ is \emph{cofinal in $\NN$}, if the following conditions are satisfied:
		\begin{enumerate}
			\item $\MM \subset \NN$.
			\item For every $N \in \NN$, there exists a $\Theta$-inflation $N\rightarrowtail M$ with $M\in\MM$.
			\item If $ 0\neq N\in\NN$, then $N$ is left Schurian for $\Sub(\MM)$.
		\end{enumerate}
		A monobrick $\MM$ is called {\em cofinally closed} if there is no proper cofinal extension of $\MM$. The set of cofinally closed monobricks in $\AA$ is denoted by $\ccmbrick\AA$.
	\end{definition}
	
	\begin{remark}\label{rem:cofinal-extension} 	Let $\MM$ and $\NN$ be two monobricks in $\AA$. The first two conditions in Definition \ref{def:cof} implies that $\MM \subset \NN\subseteq\Sub(\MM)$. Now we assume that $\AA$ is an  abelian length category and set $\Theta=l_{{\rm sim}(\AA)}$. It can be check that $N$ is left Schurian for $\Sub(\MM)$ for any $N\in\NN$. Thus our definition coincides with \cite[Definition 3.2]{Eno}. 
	\end{remark}

	\begin{lemma}\label{prop:cofunion}
		Let $\MM$ be a monobrick in $\AA$. If $\{ \NN_i \, | \, i \in I\}$ is a family of cofinal extensions of $\MM$, then $\NN:= \bigcup_{i \in I} \NN_i$ is a cofinal extension of $\MM$.
	\end{lemma}
	\begin{proof} It is suffices to prove $\NN$ is a  monobrick. Take $N_1,N_2 \in \NN$ with $N_1 \in \NN_{i_1}$ and $N_2 \in \NN_{i_2}$, and let $f \colon N_1 \to N_2$ be a non-zero morphism. By Remark \ref{rem:cofinal-extension},  we have $N_2\subseteq\Sub(\MM)$ and thus $f$ is a $\Theta$-inflation. This shows that $\NN$ is a  monobrick. 
	\end{proof}
	
For a monobrick $\MM$ in $\AA$, we denote by $\ov{\MM}$ the union of all cofinal extensions of $\MM$.

	\begin{lemma}\label{cor:closure}
		Let $\MM$ be a monobrick in $\AA$. Then $\ov{\MM}$ is a cofinal extension of $\MM$ and  $\ov{\MM}$ is cofinally closed.
	\end{lemma}
	\begin{proof} By Lemma \ref{prop:cofunion}, we know that $\ov{\MM}$ is a cofinal extension of $\MM$. Now suppose $\MM'$ is a cofinal extension  of $\ov{\MM}$. Then $\MM'$ is also a cofinal extension  of $\MM$. Hence $\MM'=\ov{\MM}$.
	\end{proof}
	
	\begin{proposition}\label{prop:cclosed-uni}  Let $\MM$ be a monobrick in $\AA$.  If $\NN$ is a cofinal extension of $\MM$ which is cofinally closed, then $\NN = \ov{\MM}$ holds.
	\end{proposition}
	\begin{proof}  It is obvious that  $\MM \subset \NN \subset \ov{\MM}\subseteq \Sub(\MM)$.  By Proposition \ref{pro-length-cat}(1), we have  $\ov{\MM}\subseteq  \Sub(\MM)\subseteq\Sub(\NN)$. On the other hand, we have $\NN\subseteq\Sub(\MM)$ and thus $\Sub(\NN)\subseteq\Sub(\MM)$. The observation above implies that $\ov{\MM}$ is a cofinal extension of $\NN$. Since $\NN$ is cofinally closed, we conclude that $\NN = \ov{\MM}$.
	\end{proof}
	By Lemma \ref{cor:closure}, we can define a map
	$$(\ov{?}) \colon \mbrick \AA \twoheadrightarrow \ccmbrick\AA:~\MM\mapsto\ov{\MM}.$$
	Proposition \ref{prop:cclosed-uni} implies that  $\ov{\MM}$ is the unique cofinal extension of a given monobrick $\MM$, which is cofinally closed. It is obvious that a monobrick $\MM$ is cofinally closed if and only if $ \MM=\ov{\MM}$. 
	
	Now, we can state the main result of this section. This generalizes \cite[Theorem 3.15]{Eno} for abelian length categories.

	\begin{theorem}\label{thm:torfproj}
	Let $((\AA,\E,\s), \Theta)$ be an extriangulated length category. Then we have the following commutative diagram and the horizontal maps are bijections.
		\[
		\begin{tikzcd}
			\torf\AA \rar["\simp", shift left] \ar[dd, bend right=75, "1"'] \dar[hookrightarrow] & \ccmbrick\AA \lar["\Filt", shift left]\dar[hookrightarrow] \ar[dd, bend left=75, "1"]\\
			\lSchur\AA \rar["\simp", shift left] \dar["\FFF", twoheadrightarrow] & \mbrick\AA \dar["(\ov{?}) ", twoheadrightarrow] \lar["\Filt", shift left]  \\
			\torf\AA \rar["\simp", shift left] & \ccmbrick\AA \lar["\Filt", shift left]
		\end{tikzcd}
		\]
	\end{theorem}
	\begin{proof}  By Theorem \ref{thm:schur-mbrick}, Lemma \ref{prop:torf-schur}, Lemma \ref{cor:closure} and Proposition \ref{prop:cclosed-uni}, it suffices to prove $\simp:\torf\AA\rightarrow\ccmbrick\AA $ and $\Filt:\ccmbrick\AA\rightarrow\torf\AA$ are mutually inverse to each other.

		Let $\FF$ be a torsion-free class in 	$(\AA, \Theta)$.  Using Theorem \ref{thm:schur-mbrick} together with  Lemma \ref{prop:torf-schur}, we know that $\simp(\FF)$  is a monobrick such that $\Filt\simp(\FF)=\FF$. Then  $\ov{\simp(\FF)}$ is a cofinal extension of $\simp(\FF)$ by Lemma \ref{cor:closure}. For any  $N\in\ov{\simp(\FF)}$, there exists a $\Theta$-inflation $N \rightarrowtail F$  such that $F\in\FF$.	This implies that $N\in\FF$. Thus we have  $\simp(\FF)\subseteq\ov{\simp(\FF)}\subseteq\FF$ and thus
		$$\Filt\simp(\FF)\subseteq\Filt\ov{\simp(\FF)}\subseteq\Filt(\FF).$$	
		Since $\FF$ is closed under $\Theta$-extensions, we have $\Filt(\FF)=\FF$. By Theorem \ref{thm:schur-mbrick}, we conclude that $\simp(\FF)=\ov{\simp(\FF)}$. The observation above implies that $\simp(\FF)\in\ccmbrick\AA$ and $\Filt\simp(\FF)=\FF$.
		
		Conversely, we take $\MM \in \ccmbrick\AA$.
		We claim that $\simp\FFF(\MM)$ is a cofinal extension $\MM$. Take any  $M\in\MM$. 
		Observe that $M$ is left Schurian for $\Sub(\MM)$. By using Lemma \ref{lem-schur1} and Proposition \ref{P-3-7}, we infer that $M$ is left Schurian for $\Filt(\Sub(\MM))=\FFF(\MM)$. Then Lemma \ref{schurian-simple}(2) implies that $\MM\subseteq\simp \FFF(\MM)$. If $X\in \simp\FFF(\MM)$, then  $X\in \FFF(\MM)=\Filt(\Sub(\MM))$. It is easily checked that $X\in\Sub(\MM)$.  Take a non-zero morphism $\alpha:P\rightarrow N$ with $P\in \simp\FFF(\MM)$ and $N\in \Sub(\MM)$. Consider the
		$\Theta$-decomposition $(i_{\alpha}, X_{\alpha},j_{\alpha})$ of $\alpha$. Then there exits a commutative diagram
		$$
		\xymatrix{
			\cocone i_{\alpha} \ar@{>->}[r]^{} & P\ar@{->>}[dr]_{i_{\alpha}} \ar[rr]^{\alpha} & &N~ \ar@{->>}[r]^{} & \cone j_{\alpha}.\\
			&   &  X_{\alpha}\ar@{>->}[ur]_{j_{\alpha}} &  &  }
		$$
		Obseve that $\cocone i_{\alpha}, X_{\alpha}\in\FFF(\MM)$. Since $P\in \simp\FFF(\MM)$, we infer that $i_\alpha$ is an isomorphism and hence $\alpha\cong j_{\alpha}$ is a $\Theta$-inflation. This implies that $P$ is left Schurian for $\Sub(\MM)$. The observation above implies that $\simp\FFF(\MM)$ is a cofinal extension $\MM$. Thus $\MM\subseteq\simp\FFF(\MM)\subseteq\ov{\MM} =\MM$. This implies that $\simp\FFF(\MM)=\ov{\MM} =\MM$. By using Theorem \ref{thm:schur-mbrick}, we infer that $\Filt\MM=\Filt\simp\FFF(\MM)=\FFF(\MM)\in\torf\AA$ and $\simp\Filt\MM=\MM$. This finishes the proof.
	\end{proof}

Dually, the torsion classes in length categories  can be classified by using epibricks. We finish this section with a straightforward example illustrating some of our results.
	
	\begin{table}[h]
	\caption{}
	\begin{tabular}{|p{2.9cm}|p{5.5cm}|p{2.8cm}|p{3.0cm}|}
		\hline
		monobrick $\MM$ & left Schur subcatgory $\aFilt \MM$  &torsion-free ? & $\ov{\MM} $ \\
		\hline
		0    &  0  &   Yes     &  itself    \\
		\hline
		$\{P_{1}\}$        &   	$\add(\{P_{1}\})$   &   Yes       & itself   \\
		\hline
		$\{P_{2}\}$        &   $\add(\{P_{2}\})$     &     No     &  $\{S_{1}[-1],P_{2}\}$  \\
		\hline
		$\{S_{3}\}$        &    $\add(\{S_{3}\})$     &   No	  &  $\{S_{2}[-1],S_{3}\}$  \\
		\hline
		$\{S_{1}[-1]\}$  &    $\add(\{S_{1}[-1]\})$   &  Yes  &  itself    \\
		\hline
		$\{I_{2}[-1]\}$   &  $\add(\{I_{2}[-1]\})$      &  No  &  $\{S_{2}[-1],I_{1}[-1]\}$    \\
		\hline
		$\{S_{2}[-1]\}$    &   $\add(\{S_{2}[-1]\})$   & Yes  & itself   \\
		\hline
		$\{S_{2}[-1],I_{2}[-1]\}$ &  $\add(\{S_{2}[-1],I_{2}[-1]\})$  &  Yes   & itself    \\
		\hline
		$\{S_{2}[-1],S_{3}\}$ & $\add(\{S_{2}[-1],S_{3}\})$   &   No  & $\{S_{2}[-1],I_{2}[-1],S_{3}\}$  \\
		\hline
		$\{S_{2}[-1],S_{1}[-1]\}$ &  $\add(\{S_{2}[-1],S_{1}[-1],I_{2}[-1]\})$   &   Yes  & itself    \\
		\hline
		$\{S_{2}[-1],P_{2}\}$&      $\add(\{S_{2}[-1],P_{2},S_{2}[-1],\})$       &   No  & $\{S_{2}[-1],S_{1}[-1],P_{2}\}$   \\
		\hline
		$\{S_{2}[-1],P_{1}\}$&      $\add(\{S_{2}[-1],P_{1}\})$       &   Yes  & itself    \\
		\hline
		$\{I_{2}[-1],S_{3}\}$ &     $\add(\{I_{2}[-1],S_{3}\})$        &   No  & $\{S_{2}[-1],I_{2}[-1],S_{3}\}$  \\
		\hline
		$\{I_{2}[-1],P_{1}\}$ &      $\add(\{I_{2}[-1],P_{1},S_{3}\})$       &   No  & $\{S_{2}[-1],I_{2}[-1],P_{1}\}$   \\
		\hline
		$\{S_{1}[-1],S_{3}\}$ &     $\add(\{S_{1}[-1],S_{3}\})$       &  No   & $\{S_{2}[-1],S_{1}[-1],S_{3}\}$     \\
		\hline
		$\{S_{1}[-1],P_{2}\}$ &      $\add(\{S_{1}[-1],P_{2}\})$       &   Yes  & itself    \\
		\hline
		$\{S_{1}[-1],P_{1}\}$ &      $\add(\{S_{1}[-1],P_{1},P_{2}\})$       &   Yes  & itself    \\
		\hline
		$\{S_{2}[-1],I_{2}[-1],S_{3}\}$ & $\add(\{S_{2}[-1],I_{2}[-1],S_{3}\})$   &   Yes  & itself  \\
		\hline
		$\{S_{2}[-1],I_{2}[-1],P_{1}\}$ & $\add(\{S_{2}[-1],I_{2}[-1],S_{3},P_{1}\})$       &   Yes  & itself     \\
		\hline
		$\{S_{2}[-1],S_{1}[-1],S_{3}\}$ & $\add(\{S_{2}[-1],I_{2}[-1],S_{1}[-1],S_{3}\})$          &  Yes   & itself      \\
		\hline
		$\{S_{2}[-1],S_{1}[-1],P_{2}\}$ & $\add(\{S_{2}[-1],I_{2}[-1],S_{3},S_{1}[-1],P_{2}\})$       &   Yes  & itself     \\
		\hline
		$\{S_{2}[-1],S_{1}[-1],P_{1}\}$ & 	$\AA$     &   Yes  &   itself     \\
		\hline
	\end{tabular}
\end{table}
	\begin{example}\label{exm:torfproj}  Let $\Lambda$ be the path algebra of the quiver $1\longrightarrow2\longrightarrow3$. The Auslander-Reiten quiver of the bounded derived category $D^{b}(\Lambda)$ is as follows:
		\begin{equation*}
			\xymatrix@!=0.5pc{
				&& S_3[-1]\ar[dr]  && S_2[-1]\ar[dr] && S_1[-1]\ar[dr] && P_1\ar[dr] && \\
				&&\cdots\cdots\quad& P_2[-1]\ar[dr]\ar[ur] && I_2[-1]\ar[ur]\ar[dr] && P_2\ar[ur]\ar[dr] && I_2\ar[dr] & \cdots\cdots\\
				&&&& P_1[-1]\ar[ur] && S_3\ar[ur] && S_2\ar[ur] && S_1}
		\end{equation*}
		Then one can check that the set  $\XX=\{S_{2}[-1],S_{1}[-1],P_{1}\}$ is a proper semibrick.  Set $\AA:=\0Filt_{D^{b}(\Lambda)}(\XX)$ and $\Theta:=l_{\XX}$. Then Theorem \ref{main0}(2) implies that $(\AA,\Theta)$ is a stable length category. In particular, we have $\Filt=\aFilt$. It is obvious that the set $\MM=\{S_2[-1],I_2[-1]\}$ is a monombrick. By Theorem \ref{thm:torfproj}, the subcategory $\aFilt \MM=\add(\MM)$ is  left Schurian. Indeed, $\add(\MM)$ is a torsion-free class and hence $\MM$ is  cofinally closed. In Table 1, we list  all monombricks $\MM$ and the corresponding left Schur subcatgories $\aFilt \MM$. 
	\end{example}

	$\mathbf{Acknowledgments.}$ This work was  supported by the National Natural Science Foundation
	of China  (Grant Nos. 12301042, 12271249, 12571042) and the the Natural Science Foundation of Zhejiang Province (Grant No.
	LZ25A010002)

	%%%%%%%%%%%%%%%%%%%%%%%%%%%%%%%%%%%%%%%%%%%%%%%%%%%%%%%%%%%%%%%%%%%%%%%%%%%%%%%%%%%%%%%%%%%%%%%%%%%%%%%%%%%%%%%%%%%%%%%%%%%%

\end{document}